\documentclass[twoside,reqno,A4]{amsart}
\usepackage{latexsym,rotate,eucal,cite}%%%%%%%%%%%%%%%
\usepackage{amsmath,amsthm,amssymb,amsxtra}%%%%%%%%%%%
\usepackage{cite}%%%%%%%%%%%%%%%%%%%%%%%%%%%%%%%%%%%%%

\newtheorem{theorem}{Theorem}[section]

\theoremstyle{definition}

\newtheorem{proposition}[theorem]{Proposition}
\theoremstyle{remark}
\newtheorem{remark}[theorem]{Remark}

\numberwithin{equation}{section}
\begin{document} %{\fns \today~submitted to~JMAA \hfill} %%%%%%
%%%%%%%%%%%%%%%%%%%%%%%%%%%%%%%%%%%%%%%%%%%%%%%%%%%%%%%%%%%%%%

\title[]
 {Minimal surfaces in a unit sphere pinched by intrinsic curvature and normal curvature}
%----------Author 1
\author [YANG Dan]{YANG Dan}

\address{School of Mathematics,
         Liaoning University,
          Shenyang,
       China}

\email{dlutyangdan@126.com}

%%%%%%%%%%%%%%%%%%%%%%%%%%%%%%%%%%%%%%%%%%%%%%%%%%%%%%%%%%%%%%
%\thanks{Email addresses: yu{\_}fu@yahoo.cn
% \ and \  zhhou@dlut.edu.cn}
%%%%%%%%%%%%%%%%%%%%%%%%%%%%%%%%%%%%%%%%%%%%%%%%%%%%%%%%%%%%%%
%\thanks{\dag~The corresponding author}zh
%\address{The Corresponding Author   \newline
%    Professor CHU Wenchang      \newline
%    Dipartimento di Matematica  \newline
%    Universit\`{a} del Salento  \newline
%    Lecce-Arnesano~P.~O.~Box~193\newline
%    73100 Lecce,~~ITALIA        \newline
%    tel 39+0832+297409          \newline
%    fax 39+0832+297594          \newline
%    Email \emph{chu.wenchang@unile.it}}

%%%%%%%%%%%%%%%%%%%%%%%%%%%%%%%%%%%%%%%%%%%%%%%%%%%%%%%%%%%%%%
%%%%%%%%%%%%%%%%%%%%%%%%%%%%%%%%%%%%%%%%%%%%%%%%%%%%%%%%%%%%%%

\begin{abstract}
We establish a nice orthonormal frame field on a closed surface
minimally immersed in a unit sphere $S^{n}$, under which the shape
operators take very simple forms. Using this frame field, we obtain
an interesting property $K+K^{N}=1$ for the Gauss curvature $K$ and
the normal curvature $K^{N}$ if the Gauss curvature is positive.
Moreover, using this property we obtain the pinching on the
intrinsic curvature and normal curvature, the pinching on the normal
curvature, respectively.
\end{abstract}

\maketitle
%%%%%%%%%%%%%%%%%%%%%%%%%%%%%%%%%%%%%%%%%%%%%%%%%%%%%%%%%%%%%%%%%
%%%%%%%%%%%%%%%%%%%%%%%%%%%%%%%%%%%%%%%%%%%%%%%%%%%%%%%%%%%%%%%%%
\markboth{Dan Yang}{Minimal Surfaces in a Unit Sphere}
\thispagestyle{empty}%%%%%%%%\setcounter{page}{0}%%%%%%%%%%%%%%%%

%%%%%%%%%%%%%%%%%%%%%%%%%%%%%%%%%%%%%%%%%%
\section{Introduction}
%%%%%%%%%%%%%%%%%%%%%%%%%%%%%%%%%%%%%%%%%%%
Let $x:M\rightarrow S^{n}$ be a minimal immersion of a closed
surface $M$ into an n-dimensional unit sphere $S^{n}$. Let
$K_{s}={2}/(s(s+1))$ for each natural number $s$. Using an idea of
Hopf and the global coordinates on $S^{n}$, Calabi \cite{calabi}
proved that if $M$ is a 2-sphere with constant Gauss curvature
$K_{s}$ and and $x$ is linearly full, then $n=2s$ and $x$ is
congruent to $s$-th standard minimal immersion. For more general
minimal immersion of 2-sphere into $S^{n}$, Chern [4, p. 38]
obtained an important equality about some local invariants by
choosing a local orthonormal frame filed on $S^{n}$. As a special
case, Chern showed that if the Gauss curvature $K$ is constant then
$K=K_{s}$. Furthermore, for general minimal immersion of a surface
into $S^{n}$, Kenmotsu \cite{kenmotsu1,kenmotsu2} also obtained an
important equality (see Theorem 1 in \cite{kenmotsu1}) by choosing
the frame filed, which generalized Chern's result for the minimal
immersion of 2-sphere into $S^{n}$. We could observe that the chosen
of the frame filed plays an important role in studying the minimal
surface. The first purpose of this paper is to establish a best
local orthonormal frame field on the closed surface minimally
immersed in $S^{n}$ with positive Gauss curvature, under which the
shape operators take the most simple forms, see Theorem \ref{thm1}
for detail.

Furthermore, using the frame field introduced in  Theorem
\ref{thm1}, we can obtain an interesting result $K+K^{N}=1$ for the
Gauss curvature $K$ and the normal curvature $K^{N}$, which means
that closed minimal surfaces immersed in $S^{n}$ with positive Gauss
curvature and flat or nowhere flat normal bundle are Wintgen ideal
surfaces, see Theorem \ref{thm4}. Wintgen ideal submanifolds are a
family of submanifolds satisfying a DDVV type inequality when the
equality holds true exactly, see \cite{Chen} for instance.

Recently, a remarkable result due to Baker and Nguyen \cite{baker}
says that codimensional two surfaces satisfying a nonlinear
curvature condition depending on normal curvature smoothly evolve by
mean curvature flow to round points. In the course of estimating the
nonlinearity in the Simons identity, the authors announced an
interesting result depending on a pointwise pinching of the
intrinsic and normal curvatures.
\begin{theorem}\label{thm77} {\rm ({\cite{baker}})}
 Suppose a two surface $M$ minimally immersed in $\mathbb{S}^{4}$
 satisfies $K^{N}\leq2|K|$. Then either\\
 $(1)$ $S=0$
and $M$ is a geodesic sphere; or\\
 $(2)$ $S\neq0$, in which
case either\\
 \indent $(a)$ $K^{N}=0$ and the surface is the
Clifford torus, or\\
\indent $(b)$  $K^{N}\neq0$ and it is the Veronese surface.
\end{theorem}
By using the frame field obtained in Theorem \ref{thm1}, we
generalize this result to the minimal surfaces in arbitrary
dimension unit sphere $S^{n}$. We prove that closed minimal surfaces
immersed in $S^{n}$ with nonnegative Gauss curvature and flat or
nowhere flat normal bundle satisfying $K^{N}\leq2K$ are geodesic
sphere,
 the Clifford torus, or the Veronese surface in $S^{4}$, see Theorem
\ref{thm7} for detail.

Based on this result, we continue to consider the next pinching
$2K\leq K^{N}\leq5K$, see Theorem \ref{thm8}. Then we study the
first pinching of normal curvature $0\leq K^{N}\leq 2/3$, see
Theorem \ref{thm5}, and the next pinching $2/3\leq K^{N}\leq 5/6$,
see Theorem \ref{thm6}. At last, we prove that closed surfaces
minimally immersed in $S^{n}$ with positive Gauss curvature and
non-zero constant normal curvature are generalized Veronese surfaces
studied by Calabi \cite{calabi} and do-Carmo-Wallach \cite{wallach}.

The paper is organized as follows. In Section 2, we introduce some
basic formulae for theory of submanifolds and establish an
orthonormal frame field on the closed surfaces minimally immersed in
a unit sphere, which is crucial to get the main theorems. In Section
3, we give some pinching theorems and their proofs.
%%%%%%%%%%%%%%%%%%%%%%%%%%%%%%%%%%%%%%%%%%%%%%%%%%%%%%%%%%%%%%%%%%%%%%%%%%%%%%%
\section{\bf Basic formulae and the frame field}\label{sec:proof1}
%%%%%%%%%%%%%%%%%%%%%%%%%%%%%%%%%%%%%%%%%%%%%%%%%%%%%%%%%%%%%%%%%%%%%%%%%%%%%%%
 Let $M$ be a closed surface immersed in a
unit sphere $S^{n}$. We identify $M$ with its immersed image, agree
on the following index ranges:
$$
1\leq i, j, k, l, m ,\cdots\leq 2;\quad 3\leq
\alpha,\beta,\gamma,\delta, \cdots \leq n;\quad 1\leq A, B, C,
D,\cdots\leq n,
$$
and use the Einstein convention. We take a local orthonormal frame
field $\{e_{1},\cdots,e_{n}\}$ in $TS^{n}$ such that, restricted to
$M$, at each point of $M$, $\{e_{1},e_{2}\}$ lies in the tangent
bundle $T(M)$ and $\{e_{3},\cdots,e_{n}\}$ in the normal bundle
$N(M)$. Let $\{\omega_{1},\cdots,\omega_{n}\}$ be the dual coframe
field of $\{e_{1},\cdots,e_{n}\}$ and $(\omega_{AB})$ the Riemannian
connection form matrix associated with
$\{\omega_{1},\cdots,\omega_{n}\}$. Then $(\omega_{ij})$ defines a
Riemannian connection in $T(M)$ and $(\omega_{\alpha\beta})$ defines
a normal connection in $N(M)$. The second fundamental form of $M$
can be expressed as
$$II=\omega_{i}\otimes\omega_{i\alpha}\otimes e_{\alpha}
=h_{ij}^{\alpha}\omega_{i}\otimes\omega_{j}\otimes e_{\alpha},$$
where
$$\omega_{i\alpha}=h_{ij}^{\alpha}\omega_{j};\qquad
h_{ij}^{\alpha}=h_{ji}^{\alpha}.
$$
Let $L^{\alpha}=(h_{ij}^{\alpha})_{2\times 2}$. We denote the square
of the norm of the second fundamental form $S$ by
$$S=\sum_{(\alpha,i,j)}(h_{ij}^{\alpha})^{2}.$$
The mean curvature vector field of $M$ is expressed as
$$
h=\frac{1}{2}\sum_{\alpha=3}^{n}(h^{\alpha}_{11}+h^{\alpha}_{22})
e_{\alpha},
$$
then $M$ is minimal if and only if $h=0$. The Riemannian curvature
tensor $\{R_{ijkl}\}$ and the normal curvature tensor
$\{R_{\alpha\beta kl}\}$ are expressed as
\begin{eqnarray}\label{eq2}
R_{ijkl}=(\delta_{ik}\delta_{jl}-\delta_{il}\delta_{jk})
+h_{ik}^{\alpha}h_{jl}^{\alpha}-h_{il}^{\alpha}h_{jk}^{\alpha},\quad
R_{\alpha\beta
kl}=h_{km}^{\alpha}h_{ml}^{\beta}-h_{lm}^{\alpha}h_{mk}^{\beta}.
\end{eqnarray}
 We
denote the normal scalar curvature $K^{N}$ by
$$K^{N}=\frac{1}{2}\sqrt{\sum_{(\alpha,\beta,i,j)}(R_{\alpha\beta ij})^{2}}.$$
 The first and the second order covariant derivatives of
$\{h_{ij}^{\alpha}\}$, say$\{h_{ijk}^{\alpha}\}$ and
$\{h_{ijkl}^{\alpha}\}$ are defined as follows:
$$\nabla h_{ij}^{\alpha}=h_{ijk}^{\alpha}\omega_{k}=dh_{ij}^{\alpha}
+h_{mj}^{\alpha}\omega_{mi}+h_{im}^{\alpha}\omega_{mj}+h_{ij}^{\beta}\omega_{\beta\alpha},$$
$$\nabla
h_{ijk}^{\alpha}=h_{ijkl}^{\alpha}\omega_{l}=dh_{ijk}^{\alpha}
+h_{mjk}^{\alpha}\omega_{mi}+h_{imk}^{\alpha}\omega_{mj}+h_{ijm}^{\alpha}\omega_{mk}+h_{ijk}^{\beta}\omega_{\beta\alpha}.$$
Then we have the  Codazzi equation
\begin{eqnarray}
h_{ijk}^{\alpha}=h_{ikj}^{\alpha},\label{abc}
\end{eqnarray}
and the Ricci's formula
\begin{eqnarray}
h_{ijkl}^{\alpha}-h_{ijlk}^{\alpha}=h_{pj}^{\alpha}R_{pikl}
+h_{ip}^{\alpha}R_{pjkl}+h_{ij}^{\beta}R_{\beta\alpha kl}.\label{a}
\end{eqnarray}
The Laplacian of $\{h_{ij}^{\alpha}\}$ and $\{h_{ijk}^{\alpha}\}$
are defined by
$$
\Delta h_{ij}^{\alpha}=h_{ijmm}^{\alpha},\quad \Delta
h_{ijk}^{\alpha}=h_{ijkmm}^{\alpha}.
$$
It follows from (\ref{abc}) and (\ref{a})  that
\begin{eqnarray}
\Delta
h_{ij}^{\alpha}=h_{mmij}^{\alpha}+h_{pi}^{\alpha}R_{pmjm}+h_{mp}^{\alpha}R_{pijm}
+ h_{mi}^{\delta}R_{\delta\alpha jm},\label{aa}
\end{eqnarray}
\begin{eqnarray}
 \Delta
h_{ijk}^{\alpha}&=& (\Delta h_{ij}^{\alpha})_{k}
+2h_{pjm}^{\alpha}R_{pikm}+2h_{ipm}^{\alpha}R_{pjkm}+h_{ijp}^{\alpha}R_{pmkm}\label{bb}\nonumber\\
&+&h_{pj}^{\alpha}R_{pikm,m}+h_{ip}^{\alpha}R_{pjkm,m}+2h_{ijm}^{\delta}R_{\delta\alpha
km} +h_{ij}^{\delta}R_{\delta\alpha km,m}.
\end{eqnarray}

 In the following we
will choose an orthonormal frame field on the closed surfaces
minimally immersed in a unit sphere, under which the shape operators
have very simple forms.

 If the normal bundle of $M$ immersed in $S^{n}$ is flat, the shape
operator $L^{\alpha}$ with respect to $e_{\alpha}$ can be
diagonalized simultaneously for $\alpha=3,\cdots,n$. Otherwise, at
least one of $h^{\beta}_{12}$ is not zero. Choosing a unit normal
vector field $\widetilde e_3=e/|e|$ where
$e=\sum_{\beta=3}^{n}h^{\beta}_{12}e_{\beta}$, and taking an
orthogonal transformation in the normal space $N_x(M)$, we have
$$
\begin{pmatrix}
\widetilde{e}_{3}\\
\\
\widetilde{e}_{4}\\
\\
\vdots \\
\\
\widetilde{e}_{n}
\end{pmatrix}
=\begin{pmatrix}
{h^{3}_{12}}{|e|^{-1}} & {h^{4}_{12}}{|e|^{-1}} & \cdots & {h^{n}_{12}}{|e|^{-1}}\\
\\
a_{43} & a_{44} & \cdots & a_{4n}  \\
\\
\vdots & \vdots &        & \vdots \\
\\
a_{n3} & a_{n4} & \cdots & a_{nn}
\end{pmatrix}
\begin{pmatrix}
e_{3} \\
\\
e_{4}\\
\\
\vdots \\
\\
e_{n}
\end{pmatrix}.
$$
Let $\widetilde{L}^{\alpha}=(\widetilde{h}^\alpha_{ij})$ be the
shape operator with respect to $\widetilde{e}_{\alpha}$. Then

$$\left\{
\begin{array}{l}
h_{ij}^{3}={h_{12}^{3}}{|e|^{-1}}\widetilde{h}_{ij}^{3}+a_{43}\widetilde{h}_{ij}^{4}
           +\cdots+a_{n3}\widetilde{h}_{ij}^{n},\\
\\
h_{ij}^{4}={h_{12}^{4}}{|e|^{-1}}\widetilde{h}_{ij}^{3}+a_{44}\widetilde{h}_{ij}^{4}
           +\cdots+a_{n4}\widetilde{h}_{ij}^{n},\\
\\
\cdots\cdots\cdots\cdots\cdots\cdots\cdots\cdots\cdots\cdots\cdots\cdots \\
\\
h_{ij}^{n}={h_{12}^{n}}{|e|^{-1}}\widetilde{h}_{ij}^{3}+a_{4n}\widetilde{h}_{ij}^{4}
           +\cdots+a_{nn}\widetilde{h}_{ij}^{n}.
\end{array}\right.\eqno{(\ast)}$$
Put $i=1$ and $j=2$ in $(\ast)$. Then $(\ast)$ has a unique solution
$$b:=\widetilde{h}_{12}^{3}=|e|\neq 0,\qquad
\widetilde{h}_{12}^{\beta}=0,\quad 4\leq\beta\leq n.$$ We denote
$\lambda^{\beta}=\tilde{h}_{11}^{\beta}$ for $\beta\geq 3$. So the
shape operators $\tilde{L}^{\alpha}$  have the following forms:
\begin{eqnarray}\label{d}
\tilde{L}^{3}=\begin{pmatrix} \lambda^{3} & b  \\b&
\nu^{3}\end{pmatrix}, \qquad
\tilde{L}^{\beta}=\begin{pmatrix}\lambda^{\beta} & 0\\0&
\nu^{\beta}\end{pmatrix},
\end{eqnarray}
where  $\beta=4,\cdots, n$.

When $M$ is a minimal surface immersed in $S^{n}$, then
\begin{eqnarray}\label{6}
\lambda^{\alpha}+\nu^{\alpha}=0,
\end{eqnarray}
for $\alpha=3,\cdots,n.$ For convenience, we still denote the new
frame filed by $\{e_{\alpha}\}$ and the corresponding second
fundamental form by $\{h_{ij}^{\alpha}\}$. We denote
$$\overline{S}:=\sum_{(i,j,\beta>3)}(h_{ij}^{\beta})^{2}
             =2\sum_{(\beta>3)} (\lambda^{\beta})^{2},\quad
S_{3}:=\sum_{(i,j)}(h_{ij}^{3})^{2}=2(\lambda^{3})^{2}+2b^{2}.$$
Then
$$S=\overline{S}+S_{3}=2\sum_{(\beta>3)}(\lambda^{\beta})^{2}+2(\lambda^{3})^{2}+2b^{2}.$$
According to (\ref{6}), we define
\begin{eqnarray}
\lambda_{i}^{\alpha}:=h_{11i}^{\alpha}=-h_{22i}^{\alpha},\qquad
\lambda_{ij}^{\alpha}:=h_{11ij}^{\alpha}=-h_{22ij}^{\alpha}.
\end{eqnarray}
 By the symmetry of
$h_{ijk}^{\alpha}$ and $h_{ijkl}^{\alpha}$ with respect to indices
$i,j,k$, we denote
\begin{eqnarray}
P&:=&\sum(h_{ijk}^{\alpha})^{2}
 =4\sum_{\alpha=3}^{n}\left((\lambda^{\alpha}_{1})^{2}+(\lambda^{\alpha}_{2})^{2}\right),\nonumber\\
Q&:=&\sum(h_{ijkl}^{\alpha})^{2}
 =4\sum_{\alpha=3}^{n}\left((\lambda^{\alpha}_{11})^{2}+(\lambda^{\alpha}_{22})^{2}
   +(\lambda^{\alpha}_{12})^{2}+(\lambda^{\alpha}_{21})^{2}\right).\label{38}
\end{eqnarray}
It follows from \eqref{eq2} and \eqref{d} that the Riemannian
curvature tensor, the normal curvature tensor and the first
covariant differentials of the normal curvature tensor become
\begin{eqnarray}\label{e}
R_{ijkl}=\left(1-S/2\right)(\delta_{ik}\delta_{jl}-\delta_{il}\delta_{jk}),\quad
 R_{3\beta 12}=-2b\lambda^{\beta},\quad
R_{\gamma\beta12}=0,
\end{eqnarray}
\begin{eqnarray}\label{f}
R_{3\beta12,k}=2(\lambda^{3}h^{\beta}_{12k}-\lambda^{\beta}h^{3}_{12k}-b\lambda^{\beta}_{k}),
\;
R_{\beta\gamma12,k}=2(\lambda^{\beta}h_{12k}^{\gamma}-\lambda^{\gamma}h_{12k}^{\beta}),
\end{eqnarray}
where $\beta,\gamma=4,\cdots,n$. It is not difficult to check that
\begin{eqnarray*}
S_{k}=2\sum h_{ij}^{\alpha}h_{ijk}^{\alpha}
=4\sum_{\beta=4}^{n}\lambda^{\beta}\lambda^{\beta}_{k}+4(\lambda^{3}\lambda^{3}_{k}
+bh^{3}_{12k}).
\end{eqnarray*}
Hence
\begin{eqnarray}
\frac{1}{4}S_{1}=\sum_{\beta=4}^{n}\lambda^{\beta}\lambda^{\beta}_{1}
                +\lambda^{3}\lambda^{3}_{1}+b\lambda^{3}_{2},\qquad
\frac{1}{4}S_{2}=\sum_{\beta=4}^{n}\lambda^{\beta}\lambda^{\beta}_{2}
                +\lambda^{3}\lambda^{3}_{2}-b\lambda^{3}_{1}.\label{h}
\end{eqnarray}
From Ricci's formula (\ref{a}), we have
\begin{eqnarray}\label{35}
\lambda^{3}_{12}-\lambda^{3}_{21}=-(2-S-\overline{S})b,\qquad
\lambda^{3}_{11}+\lambda^{3}_{22}=(2-S)\lambda^{3},\nonumber\\
\\
\lambda^{\beta}_{11}+\lambda^{\beta}_{22}=(2-S-2b^2)\lambda^{\beta},\qquad
\lambda^{\beta}_{12}-\lambda^{\beta}_{21}=-2b\lambda^{3}\lambda^{\beta},\nonumber
\end{eqnarray}
for $\beta=4,\cdots,n$.

Using the above formulae, we can obtain the following proposition
for later use.
\begin{proposition}\label{prop2}
Let $M$ be a surface minimally immersed in a unit sphere $S^{n}$.
Then
\begin{eqnarray}\label{45}
\frac{1}{2}\Delta S=P+(2-S)S-4b^{2}\overline{S}.
\end{eqnarray}
\end{proposition}

\begin{proof}
From (\ref{aa}) and (\ref{e}), we have
\begin{eqnarray*}
\sum_{i,j,\alpha}h_{ij}^{\alpha}\Delta h_{ij}^{\alpha}
&=&\sum_{i,j,p,m,\alpha}(h_{ij}^{\alpha}h_{pi}^{\alpha}R_{pmjm}
 +h_{ij}^{\alpha} h_{mp}^{\alpha}R_{pijm})
 +\sum_{i,j,m,\alpha,\delta}h_{ij}^{\alpha} h_{mi}^{\delta}R_{\delta\alpha
 jm}\nonumber\\
 &=&(2-S)(h_{ij}^{\alpha})^2+\sum_{\beta=4}^{n}4b\lambda^{\beta}R_{3\beta
12} =(2-S)S-4b^2\overline{S}.
\end{eqnarray*}
Hence
\begin{eqnarray*}
\frac{1}{2}\Delta
S=\sum_{\alpha,i,j,k}(h_{ijk}^{\alpha})^2+\sum_{\alpha,i,j}h_{ij}^{\alpha}\Delta
h_{ij}^{\alpha}=P+(2-S)S-4b^{2}\overline{S}.
\end{eqnarray*}
\end{proof}
\begin{remark}\label{rem1}
 If the normal bundle of the surface $M$ minimally immersed in $S^{n}$ is flat, we choose $e_{1}, e_{2}$
 such that $b$ is zero. We easily get
$$\frac{1}{2}\Delta S=P+(2-S)S.$$
\end{remark}
Next we consider the case that the normal bundle of $M$ is nowhere
flat. In this case, $b\neq 0$ and we can establish the following
Theorem \ref{thm1}. Some partial result was obtained in \cite{hou}.
Here, we will give the detailed proof of the theorem for the
completeness.

\begin{theorem}\label{thm1}
 Let $M$
be a surface minimally immersed in a unit sphere $S^{n}$ with
nowhere flat normal bundle.  If the Gauss curvature of $M$ is
positive, we can establish a local orthonormal frame filed
$\{e_{3},\cdots,e_{n}\}$ normal to $M$ such that the shape operators
$L^{\alpha}$ with respect to $e_{\alpha}$ have the following forms:
\begin{eqnarray*}
L^{3}= \begin{pmatrix} 0 & b  \\ b& 0\end{pmatrix}, \qquad L^{4}=
\begin{pmatrix} b & 0\\0&
-b\end{pmatrix}, \qquad L^{\beta}=
\begin{pmatrix} 0 & 0\\0&
0\end{pmatrix},
\end{eqnarray*}
where $\beta=5,\cdots,n$. Furthermore
\begin{eqnarray}\label{7}
 b^2=S/{4},\qquad
\lambda^{3}_{1}=-\lambda^{4}_{2}=-\frac{1}{4\sqrt{S}}S_{2},\qquad\lambda^{3}_{2}=\lambda^{4}_{1}=\frac{1}{4\sqrt{S}}S_{1},
\end{eqnarray}
 \begin{eqnarray}\label{12}
&&\lambda^{3}_{11}=-\lambda^{4}_{21}=-\frac{1}{4\sqrt{S}}S_{21},\quad\quad\quad\lambda^{3}_{12}=-\lambda^{4}_{22}=-\frac{1}{4\sqrt{S}}(S_{22}-P),\nonumber
\\\\
&&\lambda^{3}_{22}=\lambda^{4}_{12}=\frac{1}{4\sqrt{S}}S_{12},\qquad\qquad\lambda^{3}_{21}=\lambda^{4}_{11}=\frac{1}{4\sqrt{S}}(S_{11}-P).\nonumber
\end{eqnarray}

\end{theorem}
\begin{proof} We
take the orthonormal frame field
$\{e_{1},e_{2},e_{3},\cdots,e_{n}\}$ on $M$ such that the shape
operators have the form
\begin{eqnarray}\label{DDDDD}
L^{3}=\begin{pmatrix} \lambda^{3} & b
\\b& -\lambda^{3}\end{pmatrix}; \qquad
L^{\beta}=\begin{pmatrix}\lambda^{\beta} & 0\\0&
-\lambda^{\beta}\end{pmatrix},
\end{eqnarray}
where $ \beta=4,\cdots,n$. It is easy to check from (\ref{aa}) that
\begin{eqnarray}\label{l}
\sum_{i,j,k,\alpha}(h_{ijk}^{\alpha}\Delta h_{ij}^{\alpha})_{k} &=&
\sum_{i,j,\alpha}(\Delta h_{ij}^{\alpha})^2
+\sum_{i,j,k,l,p,\alpha}(h_{ijk}^{\alpha}h_{pik}^{\alpha}R_{pljl}
+h_{ijk}^{\alpha}h_{lpk}^{\alpha}R_{pijl})
\nonumber\\
&+&\sum_{i,j,k,l,p,\alpha}(h_{ijk}^{\alpha}h_{pi}^{\alpha}R_{pljl,k}
+h_{ijk}^{\alpha}h_{lp}^{\alpha}R_{pijl,k})
\nonumber\\
&+&\sum_{i,j,k,l,\alpha,\delta}h_{ijk}^{\alpha}h_{lik}^{\delta}R_{\delta\alpha
jl}
+\sum_{i,j,k,l,\alpha,\delta}h_{ijk}^{\alpha}h_{li}^{\delta}R_{\delta\alpha
jl,k}.
 \end{eqnarray}
Firstly, by (\ref{aa}) and (\ref{e}), we get
\begin{eqnarray*}
&&\Delta h_{11}^{\beta}=(2-S-2b^2)\lambda^{\beta}, \qquad  \Delta
h_{12}^{\beta}=2b\lambda^{3}\lambda^{\beta}, \\
&&\Delta h_{12}^{3}=(2-S-\overline{S})b, \quad\quad\quad \Delta
h_{11}^{3}=(2-S)\lambda^{3},
\end{eqnarray*}
for $\beta=4,\cdots n$, so
\begin{eqnarray}\label{m}
\sum_{i,j,\alpha}(\Delta h_{ij}^{\alpha})^2
&=&2\sum_{\beta=4}^{n}(\Delta h_{11}^{\beta})^2+2(\Delta
h_{11}^{3})^2+2\sum_{\beta=4}^{n}(\Delta h_{12}^{\beta})^2+(\Delta
h_{12}^{3})^2\nonumber\\
&=&(2-S)^2S+2(5S-8)b^2\overline{S}.
\end{eqnarray}\par\noindent
Secondly, using (\ref{e}) and (\ref{h}), we get
\begin{eqnarray}\label{n}
\sum_{i,j,k,l,p,\alpha}
\Big(h_{ijk}^{\alpha}h_{pik}^{\alpha}R_{pljl}
&+&h_{ijk}^{\alpha}h_{lpk}^{\alpha}R_{pijl}\Big)
+\sum_{i,j,k,l,\alpha,\delta} h_{ijk}^{\alpha}h_{lik}^{\delta}R_{\delta\alpha jl}\nonumber\\
=(2-S)P &+&\sum_{\gamma=4}^{n}8(\lambda^{3}_{1}\lambda^{\gamma}_{2}
-\lambda^{3}_{2}\lambda^{\gamma}_{1})R_{\gamma312}\nonumber\\
=(2-S)P
&+&16b\sum_{\gamma=4}^{n}(\lambda^{3}_{1}\lambda^{\gamma}\lambda^{\gamma}_{2}
-\lambda^{3}_{2}\lambda^{\gamma}\lambda^{\gamma}_{1}).\nonumber\\
=(2-S)P
&+&4b^2\sum_{i,j,k}(h_{ijk}^{3})^2+4b(\lambda^{3}_{1}S_{2}-\lambda^{3}_{2}S_{1}).
\end{eqnarray}
Thirdly, by the first formula of (\ref{e}), we have
\begin{eqnarray}\label{qqq}
\sum_{i,j,k,l,p,\alpha}(h_{ijk}^{\alpha}h_{pi}^{\alpha}R_{pljl,k}
+h_{ijk}^{\alpha}h_{lp}^{\alpha}R_{pijl,k})=-\sum_{i,j,k,\alpha}h_{ij}^{\alpha}h_{ijk}^{\alpha}S_{k}=-\frac{1}{2}|\nabla
S|^2.
\end{eqnarray}
At last, using (\ref{f}), we have
\begin{eqnarray}\label{s}
\sum_{i,j,k,l,\alpha,\delta}
h_{ijk}^{\alpha}h_{li}^{\delta}R_{\delta\alpha jl,k}
&=&\sum_{\gamma=4}^{n}2(\lambda^{\gamma}\lambda^{3}_{2}+b\lambda^{\gamma}_{1}
-\lambda^{3}\lambda^{\gamma}_{2})R_{3\gamma12,1}\nonumber\\
&+&\sum_{\gamma=4}^{n}2(-\lambda^{\gamma}\lambda^{3}_{1}
+b\lambda^{\gamma}_{2}+\lambda^{3}\lambda^{\gamma}_{1})R_{3\gamma12,2}\nonumber\\
&+&\sum_{\beta,\gamma=4}^{n}(-2\lambda^{\gamma}\lambda^{\beta}_{2}R_{\gamma\beta12,1}
+2\lambda^{\gamma}\lambda^{\beta}_{1}R_{\gamma\beta12,2}).\nonumber\\
&=&-\frac{1}{2}SP+4b^2\sum_{i,j,k}(h_{ijk}^{3})^2
+4b(\lambda^{3}_{1}S_{2}-\lambda^{3}_{2}S_{1})+\frac{1}{4}|\nabla
S|^2.\nonumber\\
\end{eqnarray}
Substituting (\ref{m}), (\ref{n}),(\ref{qqq}) and (\ref{s}) into
(\ref{l}), we have
\begin{eqnarray}\label{ttt} \sum_{i,j,k,\alpha}
(h_{ijk}^{\alpha}\Delta h_{ij}^{\alpha})_{k}
&=&(2-\frac{3}{2}S)P+(2-S)^2S+2(5S-8)b^2\overline{S}\nonumber\\
&+&8b^{2}\sum_{i,j,k}(h_{ijk}^{3})^2+8b(\lambda^{3}_{1}S_{2}-\lambda^{3}_{2}S_{1})-\frac{1}{4}|\nabla
S|^2,
\end{eqnarray}
which together with $P=\frac{1}{2}\Delta S-(2-S)S+4b^2\overline{S}$
and  $S\Delta S=\frac{1}{2}\Delta S^2-|\nabla S|^2$ forces that
\begin{eqnarray}\label{tt}
\sum_{i,j,k,\alpha} (h_{ijk}^{\alpha}\Delta h_{ij}^{\alpha})_{k}
&=&\frac{1}{2}(2-S)S^2+4(S-2)b^2\overline{S}+\Delta
S-\frac{3}{8}\Delta S^2\nonumber\\
&&+8b^{2}\sum_{i,j,k}(h_{ijk}^{3})^2+8b(\lambda^{3}_{1}S_{2}-\lambda^{3}_{2}S_{1})+\frac{1}{2}|\nabla
S|^2.
\end{eqnarray}
Taking integration over $M$ on both sides of (\ref{tt}), we have
\begin{eqnarray*}
\int_{M}\left\{\frac{1}{2}S^2(2-S)+4(S-2)b^2\overline{S}
  +2(4b\lambda^{3}_{1}+\frac{1}{2}S_{2})^2+2(4b\lambda^{3}_{2}-\frac{1}{2}S_{1})^2\right\}=0,
\end{eqnarray*}
that is
\begin{eqnarray}\label{u}
\int_{M}(2-S)b^2\overline{S}&=&\int_{M}\left\{\frac{1}{8}S^2(2-S)+\frac{1}{2}(4b\lambda^{3}_{1}+\frac{1}{2}S_{2})^2+\frac{1}{2}(4b\lambda^{3}_{2}-\frac{1}{2}S_{1})^2\right\}\nonumber\\
&\geq&\int_{M}\frac{1}{8}S^2(2-S),
\end{eqnarray}
and the equality holds if and only if
$$b\lambda^{3}_{1}=-\frac{1}{8}S_{2},\qquad b\lambda^{3}_{2}=\frac{1}{8}S_{1}.$$
On the other hand, it is easy to check
\begin{eqnarray}\label{v}
b^2\overline{S}\leq\frac{1}{2}S_{3}\overline{S}\leq\frac{1}{8}(S_{3}
+\overline{S})^2=\frac{1}{8}S^2,
\end{eqnarray}
and the equality holds if and only if
\begin{eqnarray*}\lambda^{3}=0,\quad \overline{S}=S_{3}=2b^2.\end{eqnarray*}
Since the Gauss curvature of $M$ is positive, we have $S<2$. Taking
integration over $M$ on both sides of (\ref{v}), we obtain
\begin{eqnarray}\label{w}
\int_{M}(2-S)b^2\overline{S}\leq\int_{M}\frac{1}{8}S^2(2-S).
\end{eqnarray}
It follows from (\ref{u}) and (\ref{w}) that
$$\int_{M}(2-S)b^2\overline{S}=\int_{M}\frac{1}{8}S^2(2-S),$$
which implies that the equalities in (\ref{u}) and (\ref{v}) hold
always. Therefore
\begin{eqnarray*}\label{x}
\lambda^{3}=0,\quad b\lambda^{3}_{1}=-\frac{1}{8}S_{2},\quad
b\lambda^{3}_{2}=\frac{1}{8}S_{1},\quad\overline{S}=2b^2.
\end{eqnarray*}
 This
together with the fact that $S=\overline{S}+2b^2$ yields
$$b^2=\frac{1}{4}S,\qquad \overline{S}=\frac{1}{2}S,\qquad\lambda^{3}_{1}=\frac{1}{4\sqrt{S}}S_{2},\qquad\lambda^{3}_{2}=\frac{1}{4\sqrt{S}}S_{1}.$$
Therefore we deduce that there must exist a number $\beta$ such that
$\lambda^{\beta}\neq 0$ for $\beta\geq4$. We choose a unit normal
vector field $\overline{e}_{4}=e/|e|$ where
$e=\sum_{\gamma=4}^{n}h_{11}^{\gamma}e_{\gamma}$, and take an
orthogonal transformation in the normal space $N_{x}(M)$:
$\overline{e}_{3}=e_{3}$ and

$$
\begin{pmatrix}
\overline{e}_{4}\\
\\
\overline{e}_{5}\\
\\
\vdots \\
\\
\overline{e}_{n}
\end{pmatrix}
=\begin{pmatrix}
{h^{4}_{11}}{|e|^{-1}} & {h^{4}_{11}}{|e|^{-1}} & \cdots & {h^{n}_{11}}{|e|^{-1}}\\
\\
b_{54} & b_{55} & \cdots & b_{5n}  \\
\\
\vdots & \vdots &        & \vdots \\
\\
b_{n4} & b_{n5} & \cdots & b_{nn}
\end{pmatrix}
\begin{pmatrix}
e_{4} \\
\\
e_{5}\\
\\
\vdots \\
\\
e_{n}
\end{pmatrix}.
$$
Let $\overline{L}^{\alpha}=(\overline{h}^\alpha_{ij})$ be the shape
operators with respect to $\overline{e}_{\alpha}$, $3\le \alpha\le
n$. It follows that
$$\left\{
\begin{array}{l}
h_{ij}^{3}=\overline{h}_{ij}^{3},\\
h_{ij}^{4}={\overline{h}_{11}^{4}}{|e|^{-1}}\overline{h}_{ij}^{4}+b_{54}\overline{h}_{ij}^{5}
           +\cdots+b_{n4}\overline{h}_{ij}^{n},\\
\\
h_{ij}^{5}={\overline{h}_{11}^{5}}{|e|^{-1}}\overline{h}_{ij}^{4}+b_{55}\overline{h}_{ij}^{5}
           +\cdots+b_{n5}\overline{h}_{ij}^{n},\\
\\
\cdots\cdots\cdots\cdots\cdots\cdots\cdots\cdots\cdots\cdots\cdots\cdots \\
\\
h_{ij}^{n}={\overline{h}_{11}^{n}}{|e|^{-1}}\overline{h}_{ij}^{4}+b_{5n}\overline{h}_{ij}^{5}
           +\cdots+b_{nn}\overline{h}_{ij}^{n}.
\end{array}\right.\eqno{(\star)}$$
Put $i=1$ and $j=1$ in $(\star)$. Then it is easy to check that
$(\star)$ has unique solution
$$\overline{h}_{11}^{4}=|e|^{-1}> 0,\qquad
\overline{h}_{11}^{\gamma}=0,\qquad 5\leq\gamma\leq n.$$ Put  $i=1$
and $j=2$ in $(\star)$. Then it is easy to check that $(\star)$ has
unique solution $\overline{h}_{12}^{\gamma}=0$, $5\leq\gamma\leq n.$
Therefore
 the shape operators with respect to   $\{e_1, e_2; \overline{e}_{\gamma}\}^n_{\gamma=3}$
 have the following forms:
\begin{eqnarray*}
\overline{L}^{3}= \begin{pmatrix} 0 & b  \\ b& 0\end{pmatrix},
\qquad \overline{L}^{4}=
\begin{pmatrix} \overline{\lambda}^{4} & 0\\0&
-\overline{\lambda}^{4}\end{pmatrix}, \qquad \overline{L}^{\beta}=
\begin{pmatrix} 0 & 0\\0&
0\end{pmatrix},\qquad 5\leq\beta\leq n,
\end{eqnarray*}
where $b^2=(\overline{\lambda}^{4})^2=S/{4}$, that is
$\overline{\lambda}^{4}=b=\sqrt{S}/2.$ For convenience, we denote
the new frame field by $\{e_1, e_2; e_{\gamma}\}^n_{\gamma=3}$. So
far, we have built a frame field on $M$ such that the shape
operators have the following forms:
\begin{eqnarray*}
L^{3}= \begin{pmatrix} 0 & b  \\ b& 0\end{pmatrix}, \qquad L^{4}=
\begin{pmatrix} b & 0\\0&
-b\end{pmatrix}, \qquad L^{\beta}=
\begin{pmatrix} 0 & 0\\0&
0\end{pmatrix},\qquad 5\leq\beta\leq n.
\end{eqnarray*}
Furthermore \begin{eqnarray}\label{14}
 b^2=S/{4},\qquad
\lambda^{3}_{1}=-\frac{1}{4\sqrt{S}}S_{2},\qquad\lambda^{3}_{2}=\frac{1}{4\sqrt{S}}S_{1}.
\end{eqnarray}
It follows from Chern \cite{chern} and the choice of the normal
vector field $e_{3},e_{4}$ that
\begin{eqnarray}\label{13}
 \sum_{\gamma=5}^{n}\lambda^{\gamma}_{1}\lambda^{\gamma}_{2}=0,\qquad
\sum_{\gamma=5}^{n}\left((\lambda^{\gamma}_{1})^2-(\lambda^{\gamma}_{2})^2\right)=0.
\end{eqnarray}
Next we take covariant differential of $h_{11}^{4}$ and have
$$h_{11k}^{4}\omega_{k}=dh_{11}^{4}+2h_{12}^{4}\omega_{21}+\sum_{\alpha=3}^{n}h_{11}^{\alpha}\omega_{\alpha4}=dh_{11}^{4}=\frac{1}{4\sqrt{S}}S_{k}\omega_{k},$$
which implies
 \begin{eqnarray}\label{15}
 \lambda^{4}_{1}= \lambda^{3}_{2}=\frac{1}{4\sqrt{S}}S_{1},\qquad\lambda^{4}_{2}=-\lambda^{3}_{1}=\frac{1}{4\sqrt{S}}S_{2}.
 \end{eqnarray}
We take covariant differential of $h_{11}^{\gamma}$ and
$h_{12}^{\gamma}$  for $5\leq \gamma\leq n$,
\begin{eqnarray*}
&&h_{11k}^{\gamma}\omega_{k}=dh_{11}^{\gamma}+2h_{12}^{\gamma}\omega_{21}+\sum_{\alpha=3}^{n}h_{11}^{\alpha}\omega_{\alpha\gamma}=b\omega_{4\gamma},\\
&&h_{12k}^{\gamma}\omega_{k}=dh_{12}^{\gamma}+h_{22}^{\gamma}\omega_{21}+h_{11}^{\gamma}\omega_{12}+\sum_{\alpha=3}^{n}h_{12}^{\alpha}\omega_{\alpha\gamma}=b\omega_{3\gamma},
 \end{eqnarray*}
which imply
\begin{eqnarray}
\omega_{3\gamma}=\frac{1}{b}(\lambda^{\gamma}_{2}\omega_{1}-\lambda^{\gamma}_{1}\omega_{2}),\qquad
\omega_{4\gamma}=\frac{1}{b}(\lambda^{\gamma}_{1}\omega_{1}+\lambda^{\gamma}_{2}\omega_{2}).
\end{eqnarray}
We take covariant differential of $h_{111}^{3}$, $h_{112}^{3}$,
$h_{111}^{4}$ and $h_{112}^{4}$ respectively
\begin{eqnarray}
&&h_{111k}^{3}\omega_{k}=dh_{111}^{3}+3h_{211}^{3}\omega_{21}+h_{111}^{4}\omega_{43}+\sum_{\gamma=5}^{n}h_{111}^{\gamma}\omega_{\gamma3},\label{16}\\
&&h_{112k}^{3}\omega_{k}=dh_{112}^{3}+2h_{212}^{3}\omega_{21}+h_{111}^{3}\omega_{12}+h_{112}^{4}\omega_{43}+\sum_{\gamma=5}^{n}h_{112}^{\gamma}\omega_{\gamma3},\label{17}\\
&&h_{111k}^{4}\omega_{k}=dh_{111}^{4}+3h_{211}^{4}\omega_{21}+h_{111}^{3}\omega_{34}+\sum_{\gamma=5}^{n}h_{111}^{\gamma}\omega_{\gamma4},\label{18}\\
&&h_{112k}^{4}\omega_{k}=dh_{112}^{4}+2h_{212}^{4}\omega_{21}+h_{111}^{4}\omega_{12}+h_{112}^{3}\omega_{34}+\sum_{\gamma=5}^{n}h_{112}^{\gamma}\omega_{\gamma4}.\label{19}
 \end{eqnarray}
Then from (\ref{16}), (\ref{19}) and use (\ref{15}), (\ref{13}) we
have
\begin{eqnarray*}
(\lambda^{3}_{11}+\lambda^{4}_{21})\omega_{1}&+&(\lambda^{3}_{12}+\lambda^{4}_{22})\omega_{2}=\sum_{\gamma=5}^{n}(\lambda_{1}^{\gamma}\omega_{\gamma3}+\lambda_{2}^{\gamma}\omega_{\gamma4})\nonumber\\
&=&-\frac{2}{b}\sum_{\gamma=5}^{n}\lambda_{1}^{\gamma}\lambda_{2}^{\gamma}\omega_{1}+\frac{1}{b}\sum_{\gamma=5}^{n}\{(\lambda_{1}^{\gamma})^2-(\lambda_{2}^{\gamma})^2\}\omega_{2}=0.
 \end{eqnarray*}
It follows from (\ref{17}), (\ref{18}), (\ref{15}) and (\ref{13})
that
\begin{eqnarray*}
(\lambda^{3}_{21}-\lambda^{4}_{11})\omega_{1}&+&(\lambda^{3}_{22}-\lambda^{4}_{12})\omega_{2}=\sum_{\gamma=5}^{n}(\lambda_{2}^{\gamma}\omega_{\gamma3}-\lambda_{1}^{\gamma}\omega_{\gamma4})\nonumber\\
&=&\frac{2}{b}\sum_{\gamma=5}^{n}\lambda_{1}^{\gamma}\lambda_{2}^{\gamma}\omega_{2}+\frac{1}{b}\sum_{\gamma=5}^{n}\left((\lambda_{1}^{\gamma})^2-(\lambda_{2}^{\gamma})^2\right)\omega_{1}=0,
 \end{eqnarray*}
therefore
\begin{eqnarray}\label{20}
\lambda^{3}_{11}+\lambda^{4}_{21}=0,\quad
\lambda^{3}_{12}+\lambda^{4}_{22}=0,\quad
\lambda^{3}_{21}-\lambda^{4}_{11}=0,\quad
\lambda^{3}_{22}-\lambda^{4}_{12}=0.
\end{eqnarray}
On the other hand, we study the second covariant differentials of
$S$. It is not difficult to check that for all $k=1,2$,
\begin{eqnarray*}
S_{kl}&=&2\sum
(h_{ijl}^{\alpha}h_{ijk}^{\alpha}+h_{ij}^{\alpha}h_{ijkl}^{\alpha})=4\sum_{\alpha}(\lambda^{\alpha}_{k}\lambda^{\alpha}_{l}+h_{12k}^{\alpha}h_{12l}^{\alpha})+b\lambda^{3}_{kl}+bh_{12kl}^{4},
\end{eqnarray*}
which is equal to
\begin{eqnarray*}
&&\lambda^{4}_{11}+\lambda^{3}_{21}=\frac{1}{2\sqrt{S}}(S_{11}-P),\qquad
\lambda^{4}_{12}+\lambda^{3}_{22}=\frac{1}{2\sqrt{S}}S_{12}\nonumber\\
 &&\lambda^{4}_{22}-\lambda^{3}_{12}=\frac{1}{2\sqrt{S}}(S_{22}-P),\qquad
\lambda^{4}_{21}-\lambda^{3}_{11}=\frac{1}{2\sqrt{S}}S_{21}.
\end{eqnarray*}
This together with (\ref{20}) gives (\ref{12}). So we complete the
proof of Theorem \ref{thm1}.
\end{proof}

 From now on we use the orthonormal frame field
established by Theorem \ref{thm1}. We conclude this section with
some interesting and elementary formulas which will be useful in the
next section. Firstly, it follows from Theorem \ref{thm1} that
$b^{2}\bar{S}=S^{2}/8$. So we can rewrite Proposition \ref{prop2} as
follows.
\begin{proposition}\label{prop3}
Suppose that $M$ is a closed surface minimally immersed in a unit
sphere $S^{n}$ with positive Gauss curvature and nowhere flat normal
bundle. We have
\begin{eqnarray}\label{eq1}
\frac{1}{2}\Delta S=P-\frac{1}{2}S(3S-4).
\end{eqnarray}
\end{proposition}

The Riemannian curvature tensor, the normal curvature tensor and the
first covariant differentials of the normal curvature tensor in
\eqref{e} and \eqref{f} can be simplified as
\begin{eqnarray}\label{0}
R_{ijkl}=\left(1-\frac{S}{2}\right)(\delta_{ik}\delta_{jl}-\delta_{il}\delta_{jk}),\quad
R_{ijkl,m}=-\frac{1}{2}S_{m}(\delta_{ik}\delta_{jl}-\delta_{il}\delta_{jk}),
\end{eqnarray}
\begin{eqnarray}\label{1}
R_{3412}=-\frac{1}{2}S,\quad
R_{3\beta12}=R_{4\beta12}=R_{\beta\gamma12}=0,
\end{eqnarray}
\begin{eqnarray}\label{2}
R_{3412,k}=-\frac{1}{2}S_{k},\quad
R_{3\beta12,1}=-2b\lambda^{\beta}_{1},\quad
R_{3\beta12,2}=-2b\lambda^{\beta}_{2},
\end{eqnarray}
\begin{eqnarray}
 R_{\beta\gamma12,k}=0,\quad
R_{4\beta12,2}=-2b\lambda^{\beta}_{1},\quad
R_{4\beta12,1}=2b\lambda^{\beta}_{2}
\end{eqnarray}
for $5\leq\beta,\gamma\leq n$. The Ricci's formula in \eqref{35}
becomes
\begin{eqnarray}\label{e3}
&&\lambda_{11}^{3}+\lambda_{22}^{3}=0,\qquad\lambda_{12}^{3}-\lambda_{21}^{3}=\dfrac
14\sqrt{S} (3S-4),\nonumber\\
&&\lambda_{11}^{\beta}+\lambda_{22}^{\beta}=0,\qquad
\lambda_{12}^{\beta}-\lambda_{21}^{\beta}=0
\end{eqnarray}
for $5\leq\beta\leq n$.
%%%%%%%%%%%%%%%%%%%%%%%%%%%%%%%%%%%%%%%%%%%%%%%%%%%%%%%%%%%%%%%%%%%%%%%%%%%%%%%
\section{\bf Main Results}
%%%%%%%%%%%%%%%%%%%%%%%%%%%%%%%%%%%%%%%%%%%%%%%%%%%%%%%%%%%%%%%%%%%%%%%%%%%%%%%
In \cite{calabi}, Calabi considered minimal immersions of compact
surfaces without boundary and with constant Gauss curvature $K$ into
$S^{n}$. He gave a complete list of all such immersions and proved
that the set of possible values of $K$ is discrete, namely
$K=K(s)=2/(s(s+1))$, $s\in \mathbb{N}$. This led to the Simon
conjecture as follows (see \cite{simon}).

{\bf{Simon conjecture} (intrinsic version)}: Let $M$ be a compact
surface minimally immersed into $S^{n}$. If $K(s+1)\leq K\leq K(s)$
for an $s\in \mathbb{N}$, then either $K=K(s+1)$ or $K=K(s)$ and the
immersion is one of the Calabi's standard minimal immersion.

There is another version of this conjecture for the extrinsic
curvature functions $S$. For minimal surfaces in $S^{n}$, both
curvature functions are related as follows:
$$2K=2-S,\qquad S=\frac{2(s-1)(s+2)}{s(s+1)}, \quad s\in \mathbb{N}.$$
Thus, for Calabi's standard immersions, we have

{\bf{Simon conjecture} (extrinsic version)}: Let $M$ be a compact
surface minimally immersed into $S^{n}$. If
$$\frac{2(s-1)(s+2)}{s(s+1)}\leq S\leq \frac{2s(s+3)}{(s+1)(s+2)},\quad s\in \mathbb{N}$$
then either $S=\frac{2(s-1)(s+2)}{s(s+1)}$ or
$S=\frac{2s(s+3)}{(s+1)(s+2)}$, and the immersion is one of the
Calabi's standard minimal immersion.

For a minimal immersion as considered above, $K=K(s)=1$ for $s=1$
gives $S=0$, and the immersion is an equator in $S^{3}(1)$.
 $K=K(s)=\frac{1}{3}$ for $s=2$ gives $S=\frac{4}{3}$ and the
immersion is a Veronese surface in $S^{4}(1)$.  $K=K(s)=\frac{1}{6}$
for $s=3$ gives $S=\frac{5}{3}$ and the immersion is a generalized
Veronese surface in $S^{6}(1)$.

So far, Simon conjecture has been solved in the case $s=1$ and
$s=2$, see \cite{L.H.B,B.K,K.M}. Using the frame field established
by Theorem \ref{thm1}, we give a very simple proof of Simon
conjecture for the minimal surface in $S^{n}$ with flat or nowhere
flat normal bundle, which is critical for later use.
\begin{theorem}\label{thm2}
 Let $M$ be a closed minimal surface
immersed in $S^{n}$ with flat or nowhere flat normal bundle. If
$0\leq S\leq {4}/{3}$, then $S=0$ or $S={4}/{3}$.
\end{theorem}
\begin{proof}
 If the normal bundle is flat, it follows from Remark \ref{rem1} and $0\leq S\leq {4}/{3}$ that
$$\frac{1}{2}\Delta S=P+(2-S)S\geq (2-S)S\geq0.$$
By integration,  we have $S=0$.

  If the normal bundle is nowhere flat,  the assumption $0\leq S \leq
 4/3$ implies that the Gauss curvature of the minimal surface is
 positive. It follows from (\ref{eq1}) in Proposition \ref{prop3} that
\begin{eqnarray}\label{29}
\frac{1}{2}\Delta
S=P-\frac{1}{2}S(3S-4)\geq-\frac{1}{2}S(3S-4)\geq0.
\end{eqnarray}
So we have $S(3S-4)=0$, it follows that
 $S=4/3$. We complete the proof of Theorem \ref{thm2}.
 \end{proof}

\begin{theorem}\label{thm3}
Let $M$ be a closed minimal surface immersed in $S^{n}$ with flat or
nowhere flat normal bundle. If ${4}/{3}\leq S \leq {5}/{3}$, then
$S={4}/{3}$ or $S={5}/{3}$.
\end{theorem}

\begin{proof}
If the normal bundle is flat, from the condition ${4}/{3}\leq S \leq
{5}/{3}$, we have $$\frac{1}{2}\Delta S=P+(2-S)S\geq (2-S)S>0.$$ By
integration, we get a contradiction.

If the normal bundle is nowhere flat, it follows from ${4}/{3}\leq S
\leq {5}/{3}$ that the Gauss curvature is positive. So we can use
the frame field introduced in Theorem \ref{thm1}. From (\ref{bb}),
we have
\begin{eqnarray}\label{a3}
\sum_{i,j,k,\alpha}h_{ijk}^{\alpha}\Delta h_{ijk}^{\alpha}
&=&\sum_{i,j,k,\alpha}(h_{ijk}^{\alpha}\Delta h_{ij}^{\alpha})_{k}
-\sum_{i,j,\alpha}(\Delta h_{ij}^{\alpha})^2+\sum_{i,j,k,p,m,\alpha}2h_{ijk}^{\alpha}h_{pj}^{\alpha}R_{pikm,m}\nonumber\\
&+&\sum_{i,j,k,p,m,\alpha}(h_{ijk}^{\alpha}h_{ijp}^{\alpha}R_{pmkm}+4h_{ijk}^{\alpha}h_{pjm}^{\alpha}R_{pikm})\nonumber\\
&+&\sum_{i,j,k,m,\alpha,\delta}2h_{ijk}^{\alpha}h_{ijm}^{\delta}R_{\delta\alpha
km}+\sum_{i,j,k,m,\alpha,\delta}h_{ijk}^{\alpha}h_{ij}^{\delta}R_{\delta\alpha
km,m}.
\end{eqnarray}
From (\ref{m}), we have
\begin{eqnarray}\label{b3}
\sum_{i,j,\alpha}(\Delta
h_{ij}^{\alpha})^2=(2-S)^2S+2(5S-8)b^2\overline{S}
=\frac{1}{4}S(3S-4)^2.
\end{eqnarray}
From (\ref{0}), we have
\begin{eqnarray}\label{10}
2\sum_{i,j,k,p,m,\alpha}h_{ijk}^{\alpha}h_{pj}^{\alpha}R_{pikm,m}=-\sum_{i,j,k,\alpha}h_{ij}^{\alpha}h_{ijk}^{\alpha}S_{k}=-\frac{1}{2}|\nabla
S|^2,
\end{eqnarray}
and
\begin{eqnarray}\label{9}
\sum_{i,j,k,p,m,\alpha}(4h_{ijk}^{\alpha}h_{pjm}^{\alpha}R_{pikm}
+h_{ijk}^{\alpha}h_{ijp}^{\alpha}R_{pmkm})=5(1-\frac{S}{2})P.
\end{eqnarray}
From (\ref{1}), we have
\begin{eqnarray}\label{dd6}
2\sum_{i,j,k,m,\alpha,\delta}h_{ijk}^{\alpha}h_{ijm}^{\delta}R_{\delta\alpha
km}=16(\lambda^{3}_{1}\lambda^{4}_{2}
-\lambda^{3}_{2}\lambda^{4}_{1})R_{4312}=-\frac{1}{2}|\nabla S|^2.
\end{eqnarray}
From (\ref{2}), we have
\begin{eqnarray*}
\sum_{i,j,k,m,\alpha,\delta}h_{ijk}^{\alpha}h_{ij}^{\delta}R_{\delta\alpha
km,m}=-\frac{1}{2}S\sum_{i,j,k}\sum_{\gamma=5}^{n}(h_{ijk}^{\gamma})^2-\frac{1}{4}|\nabla
S|^2,\label{dd3}
\end{eqnarray*}
which together with
\begin{eqnarray*}
\sum_{i,j,k}(h_{ijk}^{3})^2+\sum_{i,j,k}(h_{ijk}^{4})^2=\frac{1}{2S}|\nabla
S|^2
\end{eqnarray*}
 forces that
\begin{eqnarray}\label{11}
\sum_{i,j,k,m,\alpha,\delta}h_{ijk}^{\alpha}h_{ij}^{\delta}R_{\delta\alpha
km,m}=-\frac{1}{2}SP\label{dd5}.
\end{eqnarray}
Substituting (\ref{b3}),(\ref{10}), (\ref{9}),  (\ref{dd6}) and
(\ref{11}) into (\ref{a3}), we get
\begin{eqnarray}\label{dd}
\sum_{i,j,k,\alpha}h_{ijk}^{\alpha}\Delta h_{ijk}^{\alpha}
&=&\sum_{i,j,k,\alpha}(h_{ijk}^{\alpha}\Delta
h_{ij}^{\alpha})_{k}+\frac{5}{2}\Delta S-\frac{3}{4}\Delta
S^2\nonumber\\
&+&\frac{1}{2}|\nabla S|^2-\frac{1}{4}S(3S-4)(9S-14).
\end{eqnarray}

On the other hand, It follows from (\ref{38}) and (\ref{20}) that
$$Q=8\left((\lambda_{11}^{3})^{2}+(\lambda_{12}^{3})^{2}+(\lambda_{21}^{3})^{2}
+(\lambda_{22}^{3})^{2}\right)+4\sum_{\beta=5}^{n}\left((\lambda_{11}^{\beta})^{2}+(\lambda_{12}^{\beta})^{2}
+(\lambda_{21}^{\beta})^{2}+(\lambda_{22}^{\beta})^{2}\right).$$ It
is easy to check that the relative minimal value Q with the
constraint (\ref{e3}) is $\frac{1}{4}S(3S-4)^2$, which together with
(\ref{dd}) forces that
\begin{eqnarray}\label{gg}
\frac{1}{2}\Delta P &=& \sum_{i,j,k,\alpha}h_{ijk}^{\alpha}\Delta
h_{ijk}^{\alpha}
+\sum_{i,j,k,l,\alpha}(h_{ijkl}^{\alpha})^{2}\nonumber\\
&\geq&\sum_{i,j,k,\alpha}(h_{ijk}^{\alpha}\Delta
h_{ij}^{\alpha})_{k}+\frac{5}{2}\Delta S-\frac{3}{4}\Delta
S^2-\frac{1}{2}S(3S-4)(3S-5).
\end{eqnarray}
Taking integration over $M$ on both sides of (\ref{gg}) and using
the Stokes formula, we have
\begin{eqnarray*}
0\geq-\int_{M}\frac{1}{2}S(3S-4)(3S-5).
\end{eqnarray*}
It follows that $S={4}/{3}$ or $S={5}/{3}$ if ${4}/{3}\leq S \leq
{5}/{3}$. We complete the proof of Theorem \ref{thm3}.
\end{proof}
Now an important result can be obtained instantly as follows.
\begin{theorem}\label{thm4}
 Let $M$ be a closed minimal surface
immersed in $S^{n}$ with positive Gauss curvature $K$ and flat or
nowhere flat normal bundle. Then $K+K^{N}=1$, i.e. $M$ is a minimal
Wintgen ideal surface.
\end{theorem}
\begin{proof}
If the normal bundle is flat, since the Gauss curvature of $M$ is
positive, we have $S<2$. It follows from Remark \ref{rem1} that
$$\frac{1}{2}\Delta S=P+(2-S)S\geq (2-S)S\geq0.$$
 It follows that $S=0$ and $K=1$, so the Gauss curvature $K$ and the normal curvature
 $K^{N}$ satisfy $K+K^{N}=1$.

 If the normal bundle is nowhere flat, it follows
from formula \eqref{1} that the normal curvature $K^{N}=S/2=1-K$, so
$K+K^{N}=1$. This completes the proof of Theorem \ref{thm4}.
\end{proof}

\begin{remark}
In submanifolds theory, the famous DDVV inequality is related to the
scalar curvature, normal scalar curvature and mean curvature, which
is proved by Ge-Tang \cite{getang} and Lu \cite{Lu}, independently.
Submanifolds are called Wintgen ideal when the equality in the DDVV
inequality holds true exactly. By the equality characterization of
the DDVV inequality proved by Ge and Tang \cite{getang}, the shape
operators in Theorem \ref{thm1} could attain the equality of DDVV
inequality. In this way, Theorem \ref{thm4} can be also deduced.
\end{remark}
\begin{remark}
There are also some important properties concerning the sum of the
Gauss curvature and normal curvature for closed surfaces immersed in
space forms (see \cite{pengtang}).
\end{remark}
As we know very well, there are lots of results concerning the
pinching of the second fundamental form in a unit sphere $S^{n}$.
But there are little results concerning the pinching of the normal
curvature. In the following part, we will provide some new results
of closed surfaces immersed in arbitrary dimensional unit sphere
$S^{n}$ depending on a pinching of the intrinsic and normal
curvature.
\begin{theorem}\label{thm7}
 Let $M$ be a closed surface minimally
immersed in $S^{n}$ with flat or nowhere flat normal bundle
satisfying $K^{N}\leq2K$.\\
\indent$(1)$ If $K^{N}=0$, then  either\\
\indent\indent $(a)$ $S=0$ and the surface is the geodesic sphere, or \\
\indent\indent $(b)$ $S=2$ and the surface is the Clifford torus.\\
\indent$(2)$ If $K^{N}\neq0$, then $K^{N}=2K$ and the surface is the
Veronese surface in $S^{4}$.
\end{theorem}
\begin{proof}
If the normal bundle is flat, we have $K^{N}=0$ and the Gauss
curvature $K$ is nonnegative. It follows from Remark \ref{rem1} that
$$\frac{1}{2}\Delta S=P+(2-S)S\geq (2-S)S\geq0.$$
So either $S=0$ and the surface is  a geodesic sphere, or $S=2$ and
the surface is the Clifford torus.

 If the normal bundle is nowhere
flat, we have that the Gauss curvature $K$ is positive from
$K^{N}\leq2K$. It follows from $K+K^{N}=1$ in Theorem \ref{thm4}
that the assumption $K^{N}\leq2K$ is equivalent to
 $S\leq 4/3$. By Theorem
\ref{thm2} we get $S=4/3$ and $M$ is the Veronese surface in
$S^{4}$.
\end{proof}

\begin{remark}
Theorem \ref{thm7} generalizes Theorem \ref{thm77} obtained by Baker
and Nguyen \cite{baker} to arbitrary codimension for the surface
with flat normal bundle or nowhere flat normal bundle.
\end{remark}

\begin{theorem}\label{thm8}
 Let $M$ be a closed surface minimally
immersed in $S^{n}$ with positive Gauss curvature satisfying $2K\leq K^{N}\leq5K$, then either\\
$(1)$ $K^{N}=2K$ and it is the Veronese
surface in $S^{4}$, or\\
$(2)$ $K^{N}=5K$ and it is the generalized Veronese surface in
$S^{6}$.
\end{theorem}
\begin{proof}
It follows from $2K\leq K^{N}\leq5K$ and the Gauss curvature $K$ is
positive that the normal curvature is positive everywhere. It
follows from $K+K^{N}=1$ in Theorem \ref{thm4} that $2K\leq
K^{N}\leq5K$ is equivalent to $4/3\leq S\leq 5/3$. By Theorem
\ref{thm3} we get $S=4/3$ and $M$ is the Veronese surface in
$S^{4}$, or $S=5/3$ and $M$ is the generalized Veronese surface in
$S^{6}$.
\end{proof}

 Next we
will give some new results concerning the normal curvature.

\begin{theorem}\label{thm5}
 Let $M$ be a closed minimal surface immersed in $S^{n}$ with positive Gauss
 curvature and flat or nowhere flat normal bundle. If $0\leq K^{N}\leq2/3$, then either\\
$(1)$ $K^{N}=0$ and it is a geodesic sphere, or\\
$(2)$ $K^{N}=2/3$ and it is the Veronese surface in $S^{4}$.
\end{theorem}
\begin{proof}
Since the Gauss curvature $K$ is positive and the normal bundle is
flat or nowhere flat, it follows from Theorem \ref{thm4} that
$K+K^{N}=1$. So $0\leq K^{N}\leq2/3$ is equivalent to
 $0\leq S\leq 4/3$. By Theorem
\ref{thm2} we get $S=0$ and $M$ is a geodesic sphere, or $S=4/3$ and
$M$ is the Veronese surface in $S^{4}$.
\end{proof}
\begin{theorem}\label{thm6}
 Let $M$ be a closed minimal surface
immersed in $S^{n}$ with positive Gauss curvature. If $2/3\leq
K^{N}\leq 5/6$,
then\\
$(1)$ $K^{N}=2/3$ and $M$ is the Veronese surface in
$S^{4}$; or\\
$(2)$ $K^{N}=5/6$ and $M$ is a generalized Veronese surface in
$S^{6}$.
\end{theorem}
\begin{proof}
We observe that the condition $2/3\leq K^{N}\leq5/6$ implies that
the normal bundle is nowhere flat. It follows from Theorem
\ref{thm4} that $K+K^{N}=1$. So $0\leq K^{N}\leq2/3$ is equivalent
to
 $4/3\leq S\leq 5/3$. By Theorem
\ref{thm2} we get $S=4/3$ and $M$ is the Veronese surface in
$S^{4}$, or $S=5/3$ and $M$ is the generalized Veronese surface in
$S^{6}$.
\end{proof}
We conclude this paper with the following theorem.
\begin{theorem}\label{thm10}
 Let $M$ be a closed minimal surface
immersed in $S^{n}$ with positive Gauss curvature. If $K^{N}$ is
non-zero constant everywhere on $M$, then $K$ is constant and the
immersion is one of the generalized Veronese surfaces.
\end{theorem}

\section*{Acknowledgments}
The author was partially supported by Chern Institute of
Mathematics. The authors would like to thank the referees for their
professional suggestions about this paper which led to various
improvements.
% ------------------------------------------------------------------------

%\section*{Acknowledgments} We would like to thank the referee for giving us
%very valuable advice and comments.
% ------------------------------------------------------------------------
%%%%%%%%%%%%%%%%%%%%%%%%%%%%%%%%%%%%%%%%%%%
%%%%%%%%%%%%%%% References
%%%%%%%%%%%%%%%%%%%%%%%%%%%%%%%%%%%%%%%%%%%
\bibliographystyle{amsplain}

%%%%%%%%%%%%%%%%%%%%%%%%%%%%%%%%%%%%%%%%%%%%%%%%%%%%%%%%%%%%%%
%%%%%%%%%%%%%%%%%%%%%%%%%%%%%%%%%%%%%%%%%%%%%%%%%%%%%%%%%%%%%%
\end{document}